\documentclass[sn-mathphys]{sn-jnl}

\jyear{2021}%
\usepackage{graphicx,amsmath,amsfonts,amssymb,mathrsfs,bm,tabularx,array,enumerate,pifont,color}

\DeclareMathOperator{\diag}{diag}

\DeclareMathOperator{\argmin}{argmin}
\DeclareMathOperator{\VI}{VI}
\DeclareMathOperator{\SOL}{SOL}

\theoremstyle{thmstyleone}%
\newtheorem{theorem}{Theorem}
\theoremstyle{thmstyletwo}%
\newtheorem{remark}{Remark}%

\theoremstyle{thmstylethree}%
\newtheorem{definition}{Definition}%

\newtheorem{lemma}{Lemma}

\begin{document}

\title[Distributed Nash equilibrium seeking for constrained aggregative games]{Exponentially convergent distributed Nash equilibrium seeking for constrained aggregative games}


\author[1,2]{\fnm{Shu} \sur{Liang}}\email{sliang@tongji.edu.cn}
\author[1,2]{\fnm{Peng} \sur{Yi}}\email{yipeng@tongji.edu.cn}
\author[1,2]{\fnm{Yiguang} \sur{Hong}}\email{yghong@iss.ac.cn}
\author[3]{\fnm{Kaixiang} \sur{Peng}}\email{kaixiang@ustb.edu.cn}

\affil[1]{\orgdiv{Department of Control Science \& Engineering}, \orgname{Tongji University}, \orgaddress{\city{Shanghai}, \postcode{200092},  \country{China}}}
\affil[2]{\orgname{Shanghai Research Institute for Intelligent Autonomous Systems}, \orgaddress{\city{Shanghai}, \postcode{201210},  \country{China}}}
\affil[3]{\orgdiv{School of Automation and Electrical Engineering}, \orgname{University of Science and Technology Beijing} \orgaddress{\city{Beijing}, \postcode{100083}, \country{China}}}


\abstract{Distributed Nash equilibrium seeking of aggregative games is investigated and a continuous-time algorithm is proposed. The algorithm is designed by virtue of projected gradient play dynamics and distributed average tracking dynamics, and is applicable to games with constrained strategy sets and weight-balanced communication graphs. We obtain an exponential convergence of the proposed algorithm to the Nash equilibrium. Numerical examples illustrate the effectiveness of our methods.}




\maketitle

\section{Introduction}\label{sec1}
Distributed Nash equilibrium seeking with game-theoretic formulation and multi-agent system consideration has received research attention from the control and optimization communities, partially due to its applications in smart grids, communication networks and artificial intelligence. Various distributed algorithms for Nash equilibrium or generalized Nash equilibrium seeking have been developed, which guide a group of discrete-time or continuous-time agents to achieve the equilibrium based on local data and information exchange over a network graph \cite{Gharesifard2016Price,Ye2018Switching,Salehisadaghiani2019ADMM,Yi2019Operator,Zeng2019Cluster,Lei2020Asynchronous}.

Aggregative games have become an important type of games since the well-known Cournot duopoly model was proposed \cite{Osborne1994Course}, where the strategic interaction is clearly characterized via an aggregation term.
Recently, aggregative games have been considered in congestion control of communication networks \cite{Barrera2015Dynamic}, public environmental models \cite{Cornes2016Aggregative},
demand response management of power systems \cite{Ye2017Game}, and multiproduct-firm oligopoly \cite{Nocke2018Multiproduct}.
Because of the large-scale systems involved in these problems, seeking or computing the Nash equilibrium in a distributed manner is of practical significance.

We consider distributed Nash equilibrium seeking of aggregative games, where the aggregation information is unavailable to each local player and the communication graph can be directed with balanced weights. Similar problems have also been investigated in \cite{Koshal2016Distributed,Ye2017Game,Liang2017Distributed,Deng2019Balanced,Zhang2019Disturbances,Parise2020Distributed}.
In this work, an exponentially convergent algorithm design is proposed for the considered problem.
First, a distributed projected gradient play dynamics is designed, where we replace the global aggregation by its local estimation to calculate the gradient. Then an average tracking dynamics is augmented, where the distributed tracking signals are local parts of the aggregation. We analyze these interconnected dynamics and prove that our distributed algorithm achieves an exponential convergence to the Nash equilibrium. The contributions are as follows:
\begin{itemize}
\item A distributed Nash equilibrium seeking algorithm for aggregative game is developed. The algorithm is designed with two interconnected dynamics: a projected gradient play dynamics for equilibrium seeking and a distributed average tracking dynamics for estimation of the aggregation. The projected part can deal with local constrained strategy sets, which generalizes those in \cite{Ye2017Game,Zhang2019Disturbances}. Also, the distributed average tracking dynamics applies to weight-balanced directed graphs, which improves the algorithm in \cite{Liang2017Distributed}.

\item Exponential convergence of the proposed distributed algorithm is obtained, which is consistent with the convergence results in \cite{Yi2016Initialization,Ye2017Game,Deng2019Balanced} for unconstrained problems and is stronger than those in \cite{Yi2016Initialization,Deng2019Balanced} for constrained ones. In other words, this is a first work, to our knowledge, to propose an exponentially convergent distributed algorithm for aggregative games with local feasible constraints.
\end{itemize}

The rest of paper is organized as follows. Section 2 shows some basic
concepts and preliminary results, while Section 3 formulates the
distributed Nash equilibrium seeking problem of aggregative games. Then Section 4 presents our main results including algorithm design and analysis. Section 5 gives a numerical example to
illustrate the effectiveness of the proposed algorithm. Finally,
Section 6 gives concluding remarks.

\section{Preliminaries}
In this section, we give basic notations and related preliminary knowledge.

Denote $\mathbb{R}^n$ as the
$n$-dimensional real vector space; denote $\mathbf{1}_n =  (1,...,1)^{T} \in \mathbb{R}^n$, and $\mathbf{0}_n = (0,...,0)^{T} \in \mathbb{R}^n$. Denote $col(x_1,...,x_n) = (x_1^{T},...,x_n^{T})^{T}$ as the column vector
stacked with column vectors $x_1,...,x_n$, $\|\cdot\|$ as the Euclidean norm, and $I_n\in \mathbb{R}^{n\times n}$ as the identity matrix.   Denote $\nabla f$ as the gradient of $f$.

A set $C \subseteq \mathbb{R}^n$ is {\em convex} if $\lambda z_1
+(1-\lambda)z_2\in C$ for any $z_1, z_2 \in C$ and $0\leq \lambda \leq 1$.
For a closed convex set $C$, the {\em projection} map $P_{C}:\mathbb{R}^n \to C$ is defined as
\begin{equation*}
P_{C}(x) \triangleq \mathop{\argmin}\limits_{y\in C} \|x-y\|.
\end{equation*}
The projection map is $1$-Lipschitz continuous, i.e.,
\begin{equation*}
\|P_{C}(x)-P_{C}(y)\|\leq \|x-y\|, \quad \forall \, x, y\in \mathbb{R}^n\text{.}
\end{equation*}

A map $F:\mathbb{R}^n\rightarrow \mathbb{R}^n$ is said to be
{\em $\mu$-strongly monotone}  on a set $\Omega$ if
\begin{equation*}
(x-y)^{T}(F(x)-F(y))\geq \mu \|x - y\|^2, \quad \forall\, x,y\in\Omega.
\end{equation*}

Given a subset $\Omega \subseteq \mathbb{R}^n$ and a map $F: \Omega \to \mathbb{R}^n$, the {\em variational inequality} problem, denoted by $\VI(\Omega,F)$,
is to find a vector $x^*\in \Omega$ such that
\begin{equation*}
(y-x^*)^TF(x^*) \geq 0,\quad \forall\,y\in \Omega,
\end{equation*}
and the set of solutions to this problem is denoted by
$\SOL(\Omega,F)$ \cite{Facchinei2003Finite}. When $\Omega$ is closed and convex, the
solution of $\VI(\Omega,F)$ can be equivalently reformulated via
projection as follows:
\begin{equation*}
x\in \SOL(\Omega,F) \Leftrightarrow x = P_\Omega(x- \alpha F(x)),\,\forall\,\alpha>0.
\end{equation*}

It is known that the information exchange among agents
can be described by a graph. A graph with node set $\mathcal{V} = \{1,2,...,N\}$ and edge set $\mathcal{E}$ is written as $\mathcal{G}=(\mathcal{V},\,\mathcal{E})$ \cite{Godsil01}. The adjacency matrix of $\mathcal{G}$ can be written as $\mathcal{A} = [a_{ij}]_{N\times N}$, where $a_{ij} >0$ if $(j,i) \in \mathcal{E}$ (meaning that agent $j$ can send its information to agent $i$, or equivalently, agent $i$ can receive some information from agent $j$), and $a_{ij} =0$, otherwise. A graph is said to be strongly connected if, for any pair of vertices, there exists a sequence of intermediate vertices connected by edges. For $i \in \mathcal{V}$, the weighted in-degree and out-degree are $d_{\text{in}}^i =\sum_{j=1}^N a_{ij}$ and $d_{\text{out}}^i =\sum_{j=1}^N a_{ji}$, respectively. A graph is weight-balanced if $d_{\text{in}}^i = d_{\text{out}}^i, \forall\,i \in \mathcal{V}$. The Laplacian matrix is $L= \mathcal{D}_{\text{in}} - \mathcal{A}$, where $\mathcal{D}_{\text{in}} = \diag\{d_{\text{in}}^1, \ldots, d_{\text{in}}^N\} \in \mathbb{R}^{N \times N}$.  The following result is well known.

\begin{lemma}
{\em
Graph $\mathcal{G}$ is weight-balanced if and only if $L + L^T$ is positive semidefinite; it is strongly connected only if zero is a simple eigenvalue of $L$.}
\end{lemma}

\section{Problem Formulation}\label{nea}
Consider an $N$-player aggregative game as follows. For $i\in \mathcal{V} \triangleq \{1,...,N\}$, the $i$th player aims to minimize its cost function $J_i(x_i,x_{-i}): \Omega \to \mathbb{R}$ by choosing the local decision variable $x_i$ from a local strategy set $\Omega_i \subset \mathbb{R}^{n_i}$, where $x_{-i} \triangleq
col(x_1,...,x_{i-1},x_{i+1},...,x_{N})$, $\Omega \triangleq
\Omega_1 \times \cdots \times \Omega_N \subset \mathbb{R}^n$ and $n = \sum_{i\in\mathcal{V}} n_i$. The {\em strategy profile} of this game is $\bm{x} \triangleq col(x_1,...,x_N)\in
\Omega$. The {\em aggregation} map $\sigma: \mathbb{R}^{n}\to \mathbb{R}^m$, to specify the cost function as $J_i(x_i,x_{-i}) = \vartheta_i(x_i,\sigma(\bm{x}))$ with a function $\vartheta_i: \mathbb{R}^{n_i+m} \to \mathbb{R}$, is defined as
\begin{equation}\label{eq:sigma}
\sigma(\bm{x}) \triangleq \frac{1}{N}\sum_{i=1}^N \varphi_i(x_i)\text{,}
\end{equation}
where $\varphi_i : \mathbb{R}^{n_i} \to \mathbb{R}^m$ is a map for the local contribution to the aggregation.

The concept of Nash equilibrium is introduced as follows.

\begin{definition}\label{defn:NE}
{\em
A strategy profile $\bm{x}^*$ is said to be an {\em
Nash equilibrium} of the game if
\begin{equation}\label{eq:K}
J_i(x_i^*,x_{-i}^*) \leq J_i(y_i,x_{-i}^*), \, \forall \, y_i\in \Omega_i, \, \forall \, i\in \mathcal{V}.
\end{equation}}
\end{definition}
Condition \eqref{eq:K} means that all players simultaneously take their own
best (feasible) responses at $x^*$, where no player can further
decrease its cost function by changing its decision variable unilaterally.

We assume that the strategy sets and the cost functions are well-conditioned in the following sense.

\noindent \textbf{A1}: For any $i\in \mathcal{V}$, $\Omega_i$ is nonempty, convex and closed.\\
\noindent \textbf{A2}: For any $i\in \mathcal{V}$, the cost function $J_{i}(x_i,x_{-i})$ and the map $\varphi(x_i)$ are differentiable with respect to $x_i$.

In order to explicitly show the aggregation of the game, let us define map $G_i: \mathbb{R}^{n_i} \times \mathbb{R}^m \to \mathbb{R}^{n_i}, i\in\mathcal{V}$ as
\begin{align}\label{eq:Gi}
G_i(x_i,\eta_i) & \triangleq \nabla_{x_i} J_{i}(\cdot,x_{-i})\mid_{\sigma(\bm{x}) = \eta_i}\\
\nonumber
& = (\nabla_{x_i} \vartheta_i(\cdot,\sigma) + \frac{1}{N} \nabla_{\sigma} \vartheta_i(x_i,\cdot)^T \nabla \varphi_i)\mid_{\sigma = \eta_i}\text{.}
\end{align}
Also, let $G(\bm{x},\bm{\eta}) \triangleq col(G_1(x_1,\eta_1), ..., G_N(x_N,\eta_N))$. Clearly, $G(\bm{x},\bm{1}_N\otimes \sigma(\bm{x})) = F(\bm{x})$, where the {\em pseudo-gradient} map $F:\mathbb{R}^n \to \mathbb{R}^n$ is defined as
\begin{equation*}\label{eq:F}
F(\bm{x}) \triangleq col\{\nabla_{x_1} J_{1}(\cdot,x_{-1}),..., \nabla_{x_N} J_{N}(\cdot,x_{-N})\}.
\end{equation*}

Under \textbf{A1} and \textbf{A2}, the Nash equilibrium of the game is a solution of the variational inequality problem
$\VI(\Omega,F)$, referring to \cite{Facchinei2003Finite}. Moreover, we need the following assumptions to ensure the existence and uniqueness of the Nash equilibrium and also to facilitate algorithm design.

\noindent \textbf{A3}: The map $F(\bm{x})$ is $\mu$-strongly monotone on $\Omega$ for some constant $\mu >0$.\\
\noindent \textbf{A4}: The map $G(\bm{x},\bm{\eta})$ is $\kappa_1$-Lipschitz continuous with respect to $\bm{x} \in \Omega$ and $\kappa_2$-Lipschitz continuous with respect to $\bm{\eta}$ for some constants $\kappa_1,\kappa_2>0$. Also, for any $i\in\mathcal{V}$, $\varphi_i$ is $\kappa_3$-Lipschitz continuous on $\Omega_i$ for some constant $\kappa_3>0$.

Note that the strong monotonicity of the pseudo-gradient map $F$ has been widely adopted in the literature such as \cite{Ye2017Game,Ye2018Switching,Salehisadaghiani2019ADMM,Yi2019Operator,Deng2019Balanced,Zhang2019Disturbances,Parise2020Distributed}.

The following fundamental result is from \cite{Facchinei2003Finite}.
\begin{lemma}\label{lem:NE}
{\em
Under \textbf{A1}-\textbf{A4}, the considered game admits a unique Nash equilibrium $\bm{x}^*$.}
\end{lemma}

In the distributed design for our aggregative game, the communication topology for each player to exchange information is assumed as follows.

\noindent \textbf{A5}: The network graph $\mathcal{G}$ is strongly connected and weight-balanced.

The goal of this paper is to design a distributed algorithm to seek the
Nash equilibrium for the considered aggregative game over weight-balanced directed graph.
\section{Main Results}

In this section, we first propose our distributed algorithm and then analyze its convergence.

\subsection{Algorithm}
Our distributed continuous-time algorithm for Nash equilibrium seeking of the considered aggregative game is designed as the following differential equations:
\begin{equation}\label{eq:algorithm2}
\left\{\begin{aligned}
\dot{x}_i & = P_{\Omega_i}(x_i - \alpha G_i(x_i,\eta_i))-x_i,&& x_i(0)\in \Omega_i\\
\dot{\theta}_i & = \beta\sum_{j = 1}^Na_{ij}(\eta_j - \eta_i),&& \theta_i(0) = \bm{0}_m\\
\eta_i & =  \theta_i + \varphi_i(x_i)
\end{aligned}\right.
\end{equation}
Algorithm parameters $\alpha$ and $\beta$ satisfy
\begin{equation}\label{eq:parametersCondition}
\begin{aligned}
& 0 < \alpha < \frac{2\mu\beta\lambda_2 - 4\kappa_2\kappa_3}{\kappa^2\beta\lambda_2 + 2\mu\kappa_2\kappa_3}, \\
& \beta > \frac{2\kappa_2\kappa_3}{\mu\lambda_2},
\end{aligned}
\end{equation}
where
\begin{equation}
\kappa \triangleq \kappa_1 + \kappa_2 \cdot \kappa_3,
\end{equation}
and $\lambda_2$ is the smallest positive eigenvalue of $\frac{1}{2}(L+L^T)$ ($L$ is the Laplacian matrix).

The compact form of \eqref{eq:algorithm2} can be written as
\begin{equation}\label{eq:algorithmCompact}
\left\{\begin{aligned}
\dot{\bm{x}} & = P_{\Omega}(\bm{x} - \alpha G(\bm{x},\bm{\eta}))-\bm{x},&& \bm{x}(0)\in \Omega\\
\dot{\bm{\theta}} & = - \beta L\otimes I_m \bm{\eta},&& \bm{\theta}(0) = \bm{0}_{mN}\\
\bm{\eta} & =  \bm{\theta} + \bm{\varphi}(\bm{x})
\end{aligned}\right.
\end{equation}
where $\bm{\varphi}(\bm{x}) = col(\varphi_1(x_1), ..., \varphi_N(x_N))$. Furthermore, we can rewrite \eqref{eq:algorithmCompact} as
\begin{equation}\label{eq:algorithmCompact2}
\left\{\begin{aligned}
\dot{\bm{x}} & = P_{\Omega}(\bm{x} - \alpha  G(\bm{x},\bm{\eta}))-\bm{x},&& \bm{x}(0)\in \Omega\\
\dot{\bm{\eta}} & = - \beta L\otimes I_m\bm{\eta} + \frac{d}{dt}\bm{\varphi}(\bm{x}),&& \bm{\eta}(0) = \bm{\varphi}(\bm{x}(0))
\end{aligned}\right.
\end{equation}
The dynamics with respect to $\bm{x}$ can be regarded as distributed projected gradient play dynamics with the global aggregation $\sigma(\bm{x})$ replaced by local variables $\eta_1, ..., \eta_N$. The dynamics with respect to $\bm{\eta}$ is distributed average tracking dynamics that estimates the value of $\sigma(\bm{x})$. The design idea is similar to \cite{Ye2017Game,Liang2017Distributed}. Here, we use projection operation to deal with local feasible constraints, and replace the nonsmooth tracking dynamics in \cite{Liang2017Distributed} by this simple one to cope with weight-balanced graphs.
\subsection{Analysis}
First, we verify that the equilibrium of dynamics \eqref{eq:algorithmCompact2} coincides with the Nash equilibrium $\bm{x}^*$.
\begin{theorem}\label{thm:equilibrium}
{\em
Under \textbf{A1} - \textbf{A5}, the equilibrium of dynamics \eqref{eq:algorithmCompact2} is
\begin{equation}\label{eq:equilibrium}
\begin{bmatrix}
\bm{x}\\
\bm{\eta}
\end{bmatrix} =
\begin{bmatrix}
\bm{x}^*\\
\bm{\eta}^*
\end{bmatrix} =
\begin{bmatrix}
\bm{x}^*\\
\bm{1}_N\otimes \sigma(\bm{x}^*)
\end{bmatrix}.
\end{equation}}
\end{theorem}

\begin{proof}
The equilibrium of \eqref{eq:algorithmCompact2} should satisfy
\begin{equation*}
\begin{aligned}
\bm{0}_n & = P_{\Omega}(\bm{x} - \alpha  G(\bm{x},\bm{\eta}))-\bm{x}\\
\bm{0}_{mN} & = - L\otimes I_m\bm{\eta}
\end{aligned}
\end{equation*}
which are obtained by setting $\dot{\bm{x}}, \dot{\bm{\eta}}$ and $\frac{d}{dt}\bm{\varphi}(\bm{x})$ as zeros. Since $\mathcal{G}$ is strongly connected, $L\otimes I_m\bm{\eta} = \bm{0}$ implies $\eta_1 = \eta_2 = \cdots = \eta_N = \eta^\diamond$ for some $\eta^\diamond$ to be further determined.

Since $\mathcal{G}$ is weight-balanced, $\bm{1}_N^T L = \bm{0}_N^T$. Combining this property with dynamics \eqref{eq:algorithmCompact2} yields
\begin{equation*}
\frac{1}{N}\sum_{i=1}^N \dot{\eta}_i = \frac{d}{dt} \sigma(\bm{x}), \quad \frac{1}{N}\sum_{i=1}^N \eta_i(0) = \sigma(\bm{x}(0)).
\end{equation*}
As a result,
\begin{equation}\label{eq:identity}
\frac{1}{N}\sum_{i=1}^N \eta_i = \sigma(\bm{x}),
\end{equation}
which implies that any equilibrium pair ($\bm{x}^\diamond, \bm{1}_N\otimes \eta^\diamond)$ should also satisfy $\eta^\diamond = \sigma(\bm{x}^\diamond)$.

Substituting $\bm{x}^\diamond, \bm{1}_N\otimes \eta^\diamond$ into the projected equation for the equilibrium yields
\begin{align*}
\bm{0}_n & = P_{\Omega}(\bm{x}^\diamond - \alpha  G(\bm{x}^\diamond,\bm{1}_N\otimes\eta^\diamond))-\bm{x}^\diamond\\
& = P_{\Omega}(\bm{x}^\diamond - \alpha F(\bm{x}^\diamond))-\bm{x}^\diamond,
\end{align*}
which indicates $\bm{x}^\diamond = \bm{x}^*$. Therefore, the point given in \eqref{eq:equilibrium} is the equilibrium of \eqref{eq:algorithmCompact2}. This completes the proof.
\end{proof}

In view of the identity \eqref{eq:identity} derived from \eqref{eq:algorithmCompact2}, let
\begin{equation*}
\bm{y} \triangleq \bm{\eta} - \bm{1}_N\otimes \sigma(\bm{x}).
\end{equation*}
Then it follows from  $L \bm{1}_N = \bm{0}_N$ and \eqref{eq:algorithmCompact2} that
\begin{align}
\label{eq:subsystemX}
\dot{\bm{x}} & = P_{\Omega}\big(\bm{x} - \alpha  G(\bm{x},\bm{1}_N\otimes \sigma(\bm{x}) + \bm{y})\big)-\bm{x}\\
\label{eq:subsystemY}
\dot{\bm{y}} & = - \beta L\otimes I_m\bm{y} + \frac{d}{dt}\big(\bm{\varphi}(\bm{x}) - \bm{1}_N\otimes \sigma(\bm{x})\big)\\
\nonumber & = - \beta L\otimes I_m\bm{y} + \big(\nabla\bm{\varphi}(\bm{x}) - \bm{1}_N\otimes \nabla\sigma(\bm{x})\big)^T\cdot\\
\nonumber & \quad \quad \quad \big(P_{\Omega}\big(\bm{x} - \alpha  G(\bm{x},\bm{1}_N\otimes \sigma(\bm{x}) + \bm{y})\big)-\bm{x}\big)
\end{align}

The whole dynamics with respect to $\bm{x}$ and $\bm{y}$ consists of two interconnected subsystems as shown in Fig. \ref{fig:structure}.
\begin{figure}[h]
\begin{center}
\includegraphics[width=8.4cm]{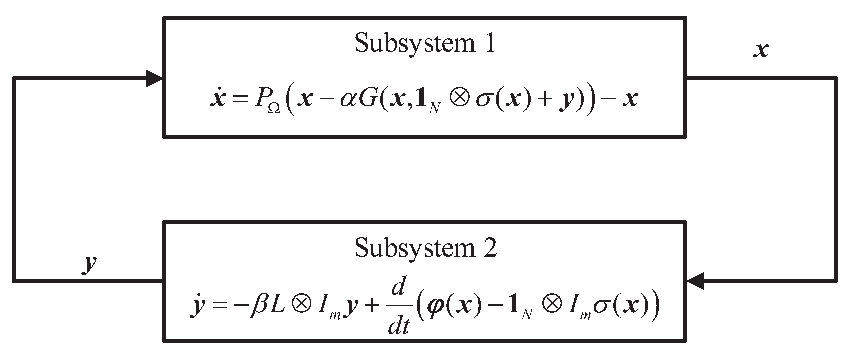}    
\caption{The interconnection of two subsystems \eqref{eq:subsystemX} and \eqref{eq:subsystemY}.} \label{fig:structure}
\end{center}
\end{figure}
Each dynamical subsystem has its own state variable, equilibrium point and external input.

Our convergence results are given in the following theorem.
\begin{theorem}\label{thm:convergence}
{\em
Under \textbf{A1}-\textbf{A5}, the distributed continuous-time algorithm \eqref{eq:algorithm2} with parameters satisfying \eqref{eq:parametersCondition} converges to the Nash equilibrium with an exponential convergence rate.
}
\end{theorem}

\begin{proof}
Let
\begin{align*}
\omega_1 & \triangleq \frac{2\alpha\cdot\mu - \alpha^2\cdot \kappa^2}{2+\alpha\cdot \kappa},\\
\omega_2 & \triangleq \beta\cdot \lambda_2 - \alpha \cdot \kappa_2 \cdot \kappa_3,\\
\xi_1 & \triangleq \alpha\cdot\kappa_2,\\
\xi_2 & \triangleq \kappa_3(2+\alpha\cdot \kappa),
\end{align*}
and
\begin{equation*}
\gamma^* \triangleq \omega_1 + \omega_2 - \sqrt{(\omega_1 - \omega_2)^2 + 4\xi_1\xi_2}.
\end{equation*}
We will show that the rate of exponential convergence of our algorithm is $\gamma^*$.
It follows from \eqref{eq:parametersCondition} that $\gamma^* > 0$, and
\begin{equation*}
(\omega_1 - \frac{\gamma^*}{2})\cdot(\omega_2 - \frac{\gamma^*}{2}) = \xi_1\xi_2
\end{equation*}

Let
\begin{align*}
H(\bm{x}) & \triangleq \bm{x} - P_{\bm{\Omega}}(\bm{x} - \alpha F(\bm{x})),\\
\widetilde{H}(\bm{x},\bm{y}) & \triangleq \bm{x} - P_{\bm{\Omega}}(\bm{x} - \alpha G(\bm{x},\bm{1}_N\otimes \sigma(\bm{x}) + \bm{y})),\\
\bm{\xi}(\bm{x},\bm{y}) & \triangleq  \widetilde{H}(\bm{x},\bm{y}) - H(\bm{x})
\end{align*}

We verify the following three properties.
\begin{enumerate}[1)]
\item $\|\bm{\xi}(\bm{x},\bm{y})\| \leq \alpha\cdot \kappa_2 \|\bm{y}\|$.
\item The map $F$ is $\kappa$-Lipschitz continuous.
\item The map $H$ is $\omega_1$-strongly monotone.
\end{enumerate}
Property 1) holds because
\begin{align*}
\|\bm{\xi}(\bm{x},\bm{y})\| & = \|\widetilde{H}(\bm{x},\bm{y}) - H(\bm{x})\| \\
& = \|P_{\bm{\Omega}}(\bm{x} - \alpha F(\bm{x})) \\
& \quad \quad - P_{\bm{\Omega}}(\bm{x} - \alpha G(\bm{x},\bm{1}_N\otimes \sigma(\bm{x}) + \bm{y}))\|\\
& \leq \alpha \|F(\bm{x}) - G(\bm{x},\bm{1}_N\otimes \sigma(\bm{x}) + \bm{y}))\|\\
& \leq \alpha\cdot \kappa_2 \|\bm{y}\|.
\end{align*}
Property 2) follows from the fact that
\begin{align*}
{}&{} \|F(\bm{y}) - F(\bm{x})\|  \\
= {}&{}\|G(\bm{y},\bm{1}_N\otimes \sigma(\bm{y})) - G(\bm{x},\bm{1}_N\otimes \sigma(\bm{x}))\| \\
\leq {}&{} \|G(\bm{y},\bm{1}_N\otimes \sigma(\bm{x}) - G(\bm{x},\bm{1}_N\otimes \sigma(\bm{x}))\|\\
{}&{} + \|G(\bm{y},\bm{1}_N\otimes \sigma(\bm{y})) - G(\bm{y},\bm{1}_N\otimes \sigma(\bm{x}))\|\\
\leq {}&{} \kappa_1 \|\bm{y} - \bm{x}\| + \kappa_2\cdot \kappa_3 \|\bm{y} - \bm{x}\|.
\end{align*}
Property 3) holds because
\begin{equation*}
\begin{aligned}
{}&{}(\bm{x} - \bm{y})^T(H(\bm{x}) - H(\bm{y}))\\
= {}&{} \|\bm{x} - \bm{y}\|^2 - (\bm{x} - \bm{y})^T\cdot\\
{}&{} \quad \quad \quad \quad \quad \quad (P_{\bm{\Omega}}(\bm{x} - \alpha F(\bm{x}))-P_{\bm{\Omega}}(\bm{y} - \alpha F(\bm{y})))\\
\geq{}&{} \|\bm{x} - \bm{y}\|(\|\bm{x} - \bm{y}\| \\
{}&{} \quad \quad \quad \quad - \|P_{\bm{\Omega}}(\bm{x} - \alpha F(\bm{x}))-P_{\bm{\Omega}}(\bm{y} - \alpha F(\bm{y}))\|)\\
\geq{}&{} \|\bm{x} - \bm{y}\|(\|\bm{x} - \bm{y}\| - \|\bm{x} - \alpha F(\bm{x})-(\bm{y} - \alpha F(\bm{y}))\|,
\end{aligned}
\end{equation*}
and
\begin{equation*}
\begin{aligned}
{}&{} \|\bm{x} - \bm{y}\| - \|\bm{x} - \alpha F(\bm{x})-(\bm{y} - \alpha F(\bm{y}))\|\\
= {}&{}\frac{\|\bm{x} - \bm{y}\|^2 - \|\bm{x} - \alpha F(\bm{x})-(\bm{y} - \alpha F(\bm{y}))\|^2}{\|\bm{x} - \bm{y}\| + \|\bm{x} - \alpha F(\bm{x})-(\bm{y} - \alpha F(\bm{y}))\|}\\
\geq {}&{} \frac{2\alpha (\bm{x} - \bm{y})^T(F(\bm{x}) - F(\bm{y})) - \alpha^2 \|F(\bm{x}) - F(\bm{y})\|^2}{(2+\alpha\cdot \kappa)\|\bm{x} - \bm{y}\|}\\
\geq {}&{} \frac{2\alpha\cdot\mu - \alpha^2\cdot \kappa^2}{2+\alpha\cdot \kappa}\|\bm{x} - \bm{y}\|.
\end{aligned}
\end{equation*}
In addition, there holds the identity $H(\bm{x}^*) = \bm{0}$, since $\bm{x}^*$ is the Nash equilibrium.

Consider the following Lyapunov candidate function
\begin{equation*}
V_1(\bm{x}) = \frac{1}{2}\|\bm{x} - \bm{x}^*\|^2.
\end{equation*}
Its time derivative along the trajectory of \eqref{eq:subsystemX} is
\begin{equation*}
\begin{aligned}
\dot{V}_1 & = - (\bm{x} - \bm{x}^*)^T \widetilde{H}(\bm{x},\bm{y})\\
& = -(\bm{x} - \bm{x}^*)^T(H(\bm{x})+ \bm{\xi}(\bm{x},\bm{y}))\\
& = -(\bm{x} - \bm{x}^*)^T(H(\bm{x}) - H(\bm{x}^*)) -(\bm{x} - \bm{x}^*)^T \bm{\xi}(\bm{x},\bm{y})\\
& \leq - \omega_1\|\bm{x} - \bm{x}^*\|^2 + \|\bm{x} - \bm{x}^*\|\|\bm{\xi}(\bm{x},\bm{y})\|\\
& \leq - \omega_1\|\bm{x} - \bm{x}^*\|^2 + \xi_1\|\bm{x} - \bm{x}^*\|\|\bm{y}\|.
\end{aligned}
\end{equation*}

Next, we focus on dynamics \eqref{eq:subsystemY}.
Let
\begin{align*}
\bm{\zeta}(\bm{x},\bm{y}) & \triangleq \frac{d}{dt}(\bm{\varphi}(\bm{x}) - \bm{1}_N\otimes \sigma(\bm{x}))\\
& = \big(\nabla\bm{\varphi}(\bm{x}) - \bm{1}_N\otimes \nabla\sigma(\bm{x})\big)^T\cdot\\
& \quad  \quad \big(P_{\Omega}\big(\bm{x} - \alpha  G(\bm{x},\bm{1}_N\otimes \sigma(\bm{x}) + \bm{y})\big)-\bm{x}\big),
\end{align*}
where the time derivative $\dot{\bm{x}}$ is along the dynamics \eqref{eq:subsystemX}.

Clearly, $\bm{1}_N^T\otimes I_m \bm{\zeta}(\bm{x},\bm{y}) = \bm{0}$. Also, since
\begin{align*}
{}&{} \|P_{\Omega}\big(\bm{x} - \alpha  G(\bm{x},\bm{1}_N\otimes \sigma(\bm{x}) + \bm{y})\big) - \bm{x}^*\|\\
\leq {}&{} \|P_{\Omega}\big(\bm{x} - \alpha  G(\bm{x},\bm{1}_N\otimes \sigma(\bm{x}) + \bm{y})\big) \\
{}&{} \quad \quad \quad \quad \quad \quad - P_{\Omega}\big(\bm{x} - \alpha  G(\bm{x},\bm{1}_N\otimes \sigma(\bm{x}))\big)\|\\
{}&{} + \|P_{\Omega}\big(\bm{x} - \alpha  F(\bm{x})\big) - P_{\Omega}\big(\bm{x}^* - \alpha  F(\bm{x}^*)\big)\|\\
\leq {}&{} \alpha \cdot \kappa_2\|\bm{y}\| + \|\bm{x} - \bm{x}^*\| + \alpha \cdot \kappa \|\bm{x} - \bm{x}^*\|,
\end{align*}
there holds
\begin{align*}
\|\bm{\zeta}(\bm{x},\bm{y})\| & \leq \|\nabla\bm{\varphi}(\bm{x}) - \bm{1}_N\otimes \nabla\sigma(\bm{x})\|\cdot\\
& \quad \quad \|P_{\Omega}\big(\bm{x} - \alpha  G(\bm{x},\bm{1}_N\otimes \sigma(\bm{x}) + \bm{y})\big)-\bm{x}\|\\
& \leq \kappa_3 \|P_{\Omega}\big(\bm{x} - \alpha  G(\bm{x},\bm{1}_N\otimes \sigma(\bm{x}) + \bm{y})\big) - \bm{x}^*\| \\
& \quad + \kappa_3\|\bm{x} - \bm{x}^*\|\\
& \leq \kappa_3\cdot(2 + \alpha\cdot \kappa)\|\bm{x} - \bm{x}^*\| + \alpha \cdot \kappa_2\cdot \kappa_3\|\bm{y}\|.
\end{align*}

Let
\begin{align*}
\widehat{\bm{y}} & \triangleq \frac{1}{N}\bm{1}_N\bm{1}_N^T\otimes I_m \bm{y},\\
\widehat{\bm{y}}^\bot & \triangleq (I_N - \frac{1}{N}\bm{1}_N\bm{1}_N^T)\otimes I_m \bm{y}.
\end{align*}
Then $\bm{y} = \widehat{\bm{y}} + \widehat{\bm{y}}^\bot$. Since $\bm{1}_N^TL = \bm{0}_N^T$, it follows from \eqref{eq:subsystemY} that
\begin{equation*}
\dot{\widehat{\bm{y}}} = \bm{0}, \quad \widehat{\bm{y}}(0) = \bm{0}.
\end{equation*}
As a result,
\begin{equation*}
\widehat{\bm{y}}(t) = \bm{0}, \quad \bm{y}(t) = \widehat{\bm{y}}^\bot(t),\quad \forall\,t\geq 0.
\end{equation*}

Consider the following Lyapunov candidate function
\begin{equation*}
V_2(\bm{y}) = \frac{1}{2}\|\bm{y}\|^2.
\end{equation*}
The time derivative of $V_2$ along the trajectory of \eqref{eq:subsystemY} is
\begin{equation*}
\begin{aligned}
\dot{V}_2 & = - \beta \bm{y}^T(L\otimes I_m)\bm{y} + \bm{y}^T \bm{\zeta}(\bm{x},\bm{y}) \\
& = - \beta \bm{y}^T \bigg(\frac{1}{2}(L+L^T)\otimes I_m\bigg)\bm{y} + \bm{y}^T \bm{\zeta}(\bm{x},\bm{y})\\
& = - \beta (\widehat{\bm{y}}^\bot)^T \bigg(\frac{1}{2}(L+L^T)\otimes I_m\bigg)\widehat{\bm{y}}^\bot + \bm{y}^T \bm{\zeta}(\bm{x},\bm{y})\\
& \leq - \beta \cdot \lambda_2 \|\widehat{\bm{y}}^\bot\|^2 + \bm{y}^T \bm{\zeta}(\bm{x},\bm{y}),
\end{aligned}
\end{equation*}
where the last inequality follows from Rayleigh quotient theorem \cite[Page 234]{Horn2013Matrix}.
Also, since $\bm{y}(t) = \widehat{\bm{y}}^\bot(t), \, \forall\,t\geq 0$,
\begin{equation*}
\begin{aligned}
\dot{V}_2 & \leq - \beta \cdot \lambda_2 \|\widehat{\bm{y}}^\bot\|^2 + \bm{y}^T \bm{\zeta}(\bm{x},\bm{y})\\
& =  - \beta\cdot \lambda_2 \|\bm{y}\|^2 + \bm{y}^T \bm{\zeta}(\bm{x},\bm{y})\\
& \leq - (\beta \cdot \lambda_2 - \alpha\cdot \kappa_2\cdot \kappa_3) \|\bm{y}\|^2 \\
& \quad + \kappa_3\cdot(2 + \alpha\cdot \kappa) \|\bm{y}\|\|\bm{x} - \bm{x}^*\|\\
& = - \omega_2 \|\bm{y}\|^2 + \xi_2 \|\bm{y}\|\|\bm{x} - \bm{x}^*\|.
\end{aligned}
\end{equation*}

Combining $V_1$ and $V_2$, let
\begin{equation*}
V \triangleq \xi_2 V_1 + \xi_1V_2
\end{equation*}
The time derivative of $V$ along the trajectory of \eqref{eq:subsystemX} and \eqref{eq:subsystemY} is
\begin{equation*}
\begin{aligned}
\dot{V} & \leq -\xi_2\omega_1\|\bm{x} - \bm{x}^*\|^2 + 2\xi_2\xi_1\|\bm{x} - \bm{x}^*\|\|\bm{y}\| - \xi_1\omega_2\|\bm{y}\|^2\\
& = - \gamma^* V - (\omega_1 - \frac{\gamma^*}{2})\xi_2\|\bm{x} - \bm{x}^*\|^2 - (\omega_2 - \frac{\gamma^*}{2})\xi_1\|\bm{y}\|^2 + 2\xi_1\xi_2\|\bm{x} - \bm{x}^*\|\|\bm{y}\|\\
& \leq - \gamma^* V.
\end{aligned}
\end{equation*}
Therefore, the algorithm converges to the Nash equilibrium with the exponential convergence rate $\gamma^*$.
\end{proof}

\begin{remark}
{\em
Exponential convergence of distributed algorithms has become a research topic in recent years. \cite{Nedich2017Achieving} has designed a distributed discrete-time optimization algorithm and proves its exponential convergence via a small-gain approach, while \cite{Liang2019Exponential} has introduced a criterion for the exponential convergence of distributed primal-dual gradient algorithms in either continuous or discrete time. Theorem \ref{thm:convergence} provides an exponential convergence result by analyzing the interconnected subsystems.
}
\end{remark}

\section{Numerical Example}
Consider a Cournot game played by $N = 20$ competitive players. For $i\in\mathcal{V}=\{1,...,N\}$, the cost function $\vartheta_i(x_i,\sigma)$ and strategy set $\Omega_i$ are
\begin{align*}
\vartheta_i(x_i,\sigma) & = a_ix_i^2 + b_ix_i + c_ix_i\sigma(\bm{x}),\\
\Omega_i & = \bigg[ - 1 - \frac{1}{2i}, \, \frac{i}{10} + \frac{1}{\sqrt{i}}\bigg],
\end{align*}
where
\begin{align*}
a_i & = 0.1 + 0.01*\sin(i), \quad b_i = \frac{i - \ln(i)}{1 + i + i^3},\\
c_i & = 0.003*\cos(i),
\end{align*}
and
\begin{equation*}
\sigma(\bm{x}) = \frac{1}{N}\sum_{j=1}^Nx_j.
\end{equation*}
It can be verified that the game mode satisfies \textbf{A1}-\textbf{A4} with constants $\mu = 0.1770, \kappa_1 = 0.2199, \kappa_2 = 0.0030, \kappa_3 = 1$.
We adopt a network graph as shown in Fig. \ref{fig:topology}, which satisfies \textbf{A5}.
\begin{figure}[h]
\begin{center}
\includegraphics[width=7cm]{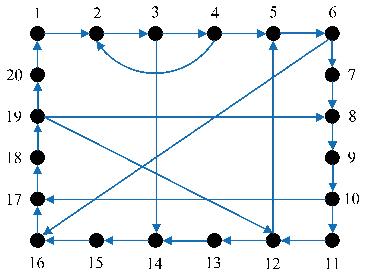}    
\caption{The communication graph of the agents.}\label{fig:topology}
\end{center}
\end{figure}

To render condition \eqref{eq:parametersCondition}, we assign $\alpha = 3$ and $\beta = 1$. The trajectory of strategy profile generated by our algorithm is shown in Fig. \ref{fig:trajectory}.
\begin{figure}[h]
\begin{center}
\includegraphics[width=8.4cm]{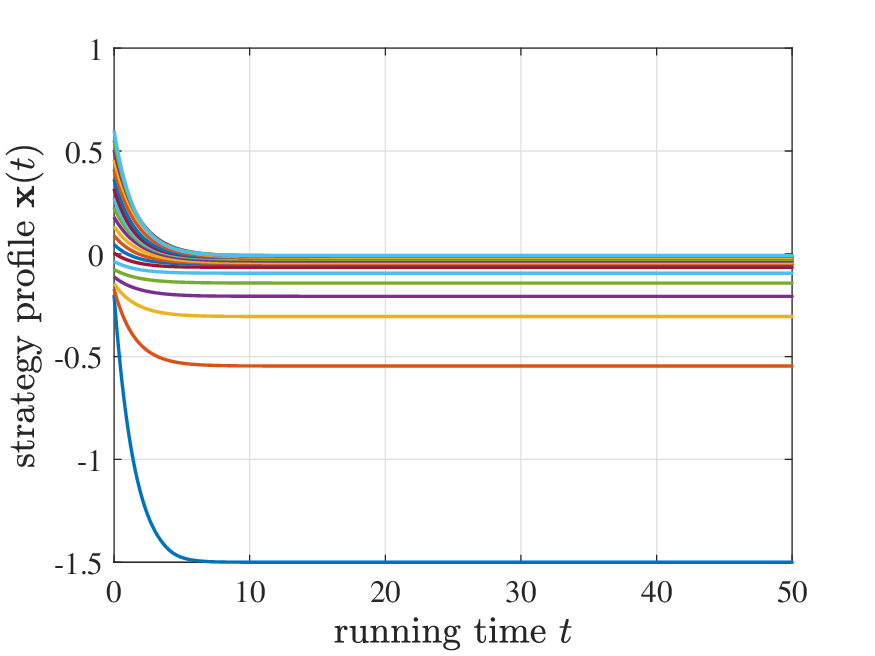}    
\caption{The trajectory of strategy profile generated by our distributed algorithm.}\label{fig:trajectory}
\end{center}
\end{figure}

In order to make some comparisons, we also use directed cycle graph and undirected Erdos-Renyi (ER) graph for the algorithm. The performance of the algorithm with these graphs is shown in Fig. \ref{fig:error}.
\begin{figure}[h]
\begin{center}
\includegraphics[width=8.4cm]{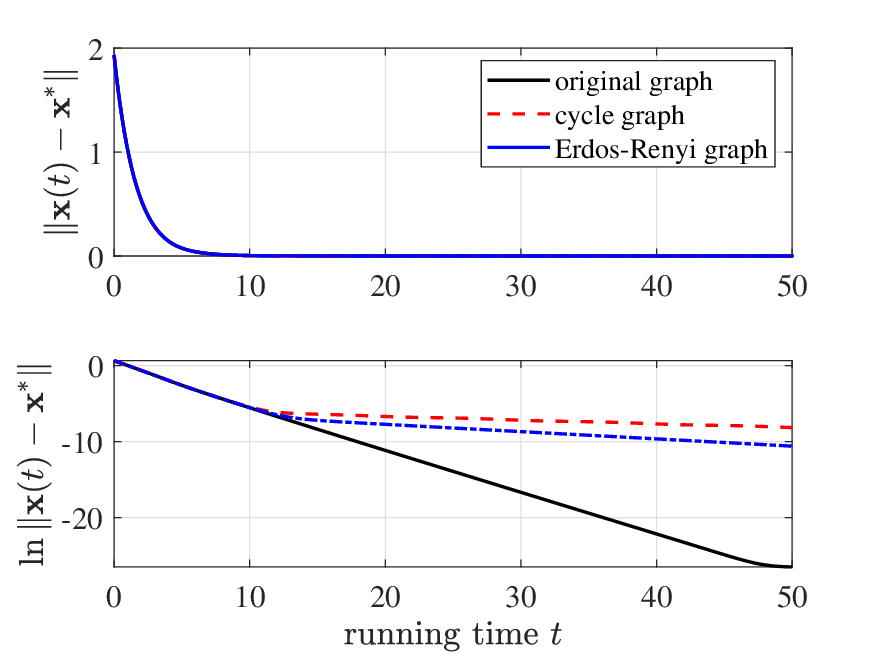}    
\caption{Performance of the algorithm with different graphs.}\label{fig:error}
\end{center}
\end{figure}

These results indicate that our distributed algorithm exponentially converges to the Nash equilibrium.

Finally, we use the undirected ER graph to compare our algorithm with the one given in \cite{Liang2017Distributed}. The numerical results are shown in Fig. \ref{fig:comparison}. It indicates that our algorithm converges faster than that algorithm. In addition, only our algorithm applies to directed graphs such as the original graph and the directed cycle graph.
\begin{figure}[h]
\begin{center}
\includegraphics[width=8.4cm]{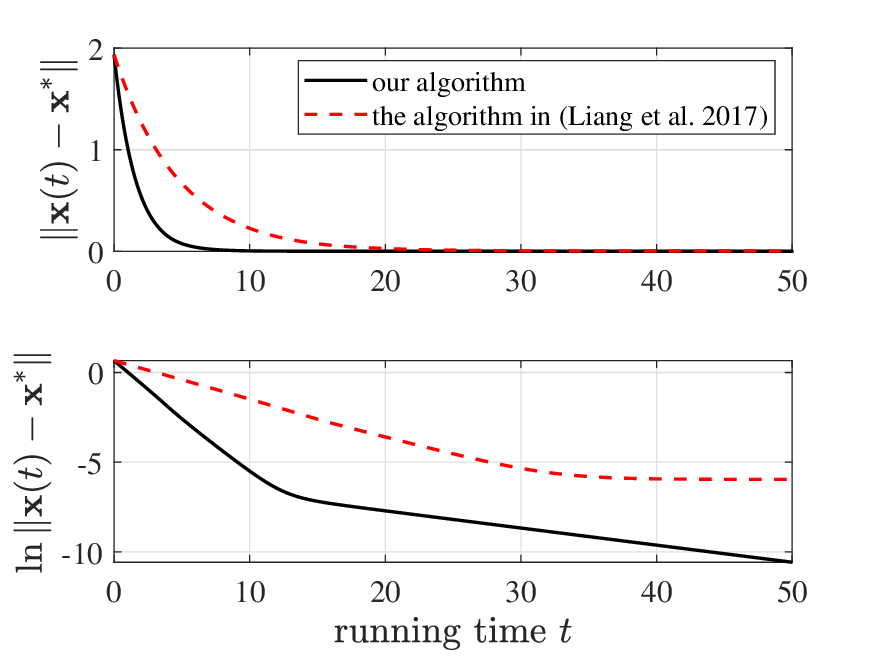}    
\caption{Performance comparison of the two distributed algorithms.}\label{fig:comparison}
\end{center}
\end{figure}
\section{Conclusions}
A distributed algorithm has been proposed for Nash equilibrium seeking of aggregative games, where the strategy set can be constrained and the network is described by a weight-balanced graph. The exponential convergence has been established.
The effectiveness of our method has also been illustrated by a numerical example.
Further work may consider generalized Nash equilibrium seeking problem for aggregative games with coupled constraints.

\section*{Declarations}
The authors confirm that there are no known conflicts of interest associated with this publication and there has been no
significant financial support for this work that could have influenced its outcome.

We confirm that the manuscript has been read and approved by all named authors and that
there are no other persons who satisfied the criteria for authorship but are not listed. We
further confirm that the order of authors listed in the manuscript has been approved by all of
us.

This paper was supported in part by National Natural Science Foundation of China
under Grant 61903027,72171171,62003239, and in part by Shanghai Municipal Science and Technology Major
Project under grant 2021SHZDZX0100, and in part by Shanghai Sailing Program under Grant Nos. 20YF1453000.

\bibliography{refference0}


\begin{thebibliography}{22}
\ifx \bisbn   \undefined \def \bisbn  #1{ISBN #1}\fi
\ifx \binits  \undefined \def \binits#1{#1}\fi
\ifx \bauthor  \undefined \def \bauthor#1{#1}\fi
\ifx \batitle  \undefined \def \batitle#1{#1}\fi
\ifx \bjtitle  \undefined \def \bjtitle#1{#1}\fi
\ifx \bvolume  \undefined \def \bvolume#1{\textbf{#1}}\fi
\ifx \byear  \undefined \def \byear#1{#1}\fi
\ifx \bissue  \undefined \def \bissue#1{#1}\fi
\ifx \bfpage  \undefined \def \bfpage#1{#1}\fi
\ifx \blpage  \undefined \def \blpage #1{#1}\fi
\ifx \burl  \undefined \def \burl#1{\textsf{#1}}\fi
\ifx \doiurl  \undefined \def \doiurl#1{\url{https://doi.org/#1}}\fi
\ifx \betal  \undefined \def \betal{\textit{et al.}}\fi
\ifx \binstitute  \undefined \def \binstitute#1{#1}\fi
\ifx \binstitutionaled  \undefined \def \binstitutionaled#1{#1}\fi
\ifx \bctitle  \undefined \def \bctitle#1{#1}\fi
\ifx \beditor  \undefined \def \beditor#1{#1}\fi
\ifx \bpublisher  \undefined \def \bpublisher#1{#1}\fi
\ifx \bbtitle  \undefined \def \bbtitle#1{#1}\fi
\ifx \bedition  \undefined \def \bedition#1{#1}\fi
\ifx \bseriesno  \undefined \def \bseriesno#1{#1}\fi
\ifx \blocation  \undefined \def \blocation#1{#1}\fi
\ifx \bsertitle  \undefined \def \bsertitle#1{#1}\fi
\ifx \bsnm \undefined \def \bsnm#1{#1}\fi
\ifx \bsuffix \undefined \def \bsuffix#1{#1}\fi
\ifx \bparticle \undefined \def \bparticle#1{#1}\fi
\ifx \barticle \undefined \def \barticle#1{#1}\fi
\bibcommenthead
\ifx \bconfdate \undefined \def \bconfdate #1{#1}\fi
\ifx \botherref \undefined \def \botherref #1{#1}\fi
\ifx \url \undefined \def \url#1{\textsf{#1}}\fi
\ifx \bchapter \undefined \def \bchapter#1{#1}\fi
\ifx \bbook \undefined \def \bbook#1{#1}\fi
\ifx \bcomment \undefined \def \bcomment#1{#1}\fi
\ifx \oauthor \undefined \def \oauthor#1{#1}\fi
\ifx \citeauthoryear \undefined \def \citeauthoryear#1{#1}\fi
\ifx \endbibitem  \undefined \def \endbibitem {}\fi
\ifx \bconflocation  \undefined \def \bconflocation#1{#1}\fi
\ifx \arxivurl  \undefined \def \arxivurl#1{\textsf{#1}}\fi
\csname PreBibitemsHook\endcsname

\bibitem{Gharesifard2016Price}
\begin{barticle}
\bauthor{\bsnm{Gharesifard}, \binits{B.}},
\bauthor{\bsnm{Ba{\c{s}}ar}, \binits{T.}},
\bauthor{\bsnm{Dominguez-Garcia}, \binits{A.D.}}:
\batitle{Price-based coordinated aggregation of networked distributed energy
  resources}.
\bjtitle{IEEE Transactions on Automatic Control}
\bvolume{61}(\bissue{10}),
\bfpage{2936}--\blpage{2946}
(\byear{2016})
\end{barticle}
\endbibitem

\bibitem{Ye2018Switching}
\begin{barticle}
\bauthor{\bsnm{Ye}, \binits{M.}},
\bauthor{\bsnm{Hu}, \binits{G.}}:
\batitle{Distributed {N}ash equilibrium seeking in multiagent games under
  switching communication topologies}.
\bjtitle{IEEE Transactions on Cybernetics}
\bvolume{48}(\bissue{11}),
\bfpage{3208}--\blpage{3217}
(\byear{2018})
\end{barticle}
\endbibitem

\bibitem{Salehisadaghiani2019ADMM}
\begin{barticle}
\bauthor{\bsnm{Salehisadaghiani}, \binits{F.}},
\bauthor{\bsnm{Shi}, \binits{W.}},
\bauthor{\bsnm{Pavel}, \binits{L.}}:
\batitle{Distributed {N}ash equilibrium seeking under partial-decision
  information via the alternating direction method of multipliers}.
\bjtitle{Automatica}
\bvolume{103},
\bfpage{27}--\blpage{35}
(\byear{2019})
\end{barticle}
\endbibitem

\bibitem{Yi2019Operator}
\begin{barticle}
\bauthor{\bsnm{Yi}, \binits{P.}},
\bauthor{\bsnm{Pavel}, \binits{L.}}:
\batitle{An operator splitting approach for distributed generalized {N}ash
  equilibria computation}.
\bjtitle{Automatica}
\bvolume{102},
\bfpage{111}--\blpage{121}
(\byear{2019})
\end{barticle}
\endbibitem

\bibitem{Zeng2019Cluster}
\begin{barticle}
\bauthor{\bsnm{Zeng}, \binits{X.}},
\bauthor{\bsnm{Chen}, \binits{J.}},
\bauthor{\bsnm{Liang}, \binits{S.}},
\bauthor{\bsnm{Hong}, \binits{Y.}}:
\batitle{Generalized {N}ash equilibrium seeking strategy for distributed
  nonsmooth multi-cluster game}.
\bjtitle{Automatica}
\bvolume{103},
\bfpage{20}--\blpage{26}
(\byear{2019})
\end{barticle}
\endbibitem

\bibitem{Lei2020Asynchronous}
\begin{barticle}
\bauthor{\bsnm{Lei}, \binits{J.}},
\bauthor{\bsnm{Shanbhag}, \binits{U.V.}}:
\batitle{Asynchronous schemes for stochastic and misspecified potential games
  and nonconvex optimization}.
\bjtitle{Operations Research}
\bvolume{68}(\bissue{6}),
\bfpage{1742}--\blpage{1766}
(\byear{2020})
\end{barticle}
\endbibitem

\bibitem{Osborne1994Course}
\begin{bbook}
\bauthor{\bsnm{Osborne}, \binits{M.J.}},
\bauthor{\bsnm{Rubinstein}, \binits{A.}}:
\bbtitle{A Course in Game Theory}.
\bpublisher{MIT Press},
\blocation{Cambridge, MA}
(\byear{1994})
\end{bbook}
\endbibitem

\bibitem{Barrera2015Dynamic}
\begin{barticle}
\bauthor{\bsnm{Barrera}, \binits{J.}},
\bauthor{\bsnm{Garcia}, \binits{A.}}:
\batitle{Dynamic incentives for congestion control}.
\bjtitle{IEEE Transactions on Automatic Control}
\bvolume{60}(\bissue{2}),
\bfpage{299}--\blpage{310}
(\byear{2015})
\end{barticle}
\endbibitem

\bibitem{Cornes2016Aggregative}
\begin{barticle}
\bauthor{\bsnm{Cornes}, \binits{R.}}:
\batitle{Aggregative environmental games}.
\bjtitle{Environmental \& Resource Economics}
\bvolume{63}(\bissue{2}),
\bfpage{339}--\blpage{365}
(\byear{2016})
\end{barticle}
\endbibitem

\bibitem{Ye2017Game}
\begin{barticle}
\bauthor{\bsnm{Ye}, \binits{M.}},
\bauthor{\bsnm{Hu}, \binits{G.}}:
\batitle{Game design and analysis for price-based demand response: an aggregate
  game approach}.
\bjtitle{IEEE Transactions on Cybernetics}
\bvolume{47}(\bissue{3}),
\bfpage{720}--\blpage{730}
(\byear{2017})
\end{barticle}
\endbibitem

\bibitem{Nocke2018Multiproduct}
\begin{barticle}
\bauthor{\bsnm{Nocke}, \binits{V.}},
\bauthor{\bsnm{Schutz}, \binits{N.}}:
\batitle{Multiproduct-firm oligopoly: an aggregative games approach}.
\bjtitle{Econometrica}
\bvolume{86}(\bissue{2}),
\bfpage{523}--\blpage{557}
(\byear{2018})
\end{barticle}
\endbibitem

\bibitem{Koshal2016Distributed}
\begin{barticle}
\bauthor{\bsnm{Koshal}, \binits{J.}},
\bauthor{\bsnm{Nedi{\'c}}, \binits{A.}},
\bauthor{\bsnm{Shanbhag}, \binits{U.V.}}:
\batitle{Distributed algorithms for aggregative games on graphs}.
\bjtitle{Operations Research}
\bvolume{63}(\bissue{3}),
\bfpage{680}--\blpage{704}
(\byear{2016})
\end{barticle}
\endbibitem

\bibitem{Liang2017Distributed}
\begin{barticle}
\bauthor{\bsnm{Liang}, \binits{S.}},
\bauthor{\bsnm{Yi}, \binits{P.}},
\bauthor{\bsnm{Hong}, \binits{Y.}}:
\batitle{Distributed {N}ash equilibrium seeking for aggregative games with
  coupled constraints}.
\bjtitle{Automatica}
\bvolume{85}(\bissue{11}),
\bfpage{179}--\blpage{185}
(\byear{2017})
\end{barticle}
\endbibitem

\bibitem{Deng2019Balanced}
\begin{barticle}
\bauthor{\bsnm{Deng}, \binits{Z.}},
\bauthor{\bsnm{Nian}, \binits{X.}}:
\batitle{Distributed generalized {N}ash equilibrium seeking algorithm design
  for aggregative games over weight-balanced digraphs}.
\bjtitle{IEEE Transactions on Neural Networks and Learning Systems}
\bvolume{30}(\bissue{3}),
\bfpage{695}--\blpage{706}
(\byear{2019})
\end{barticle}
\endbibitem

\bibitem{Zhang2019Disturbances}
\begin{barticle}
\bauthor{\bsnm{Zhang}, \binits{Y.}},
\bauthor{\bsnm{Liang}, \binits{S.}},
\bauthor{\bsnm{Wang}, \binits{X.}},
\bauthor{\bsnm{Ji}, \binits{H.}}:
\batitle{Distributed {N}ash equilibrium seeking for aggregative games with
  nonlinear dynamics under external disturbances}.
\bjtitle{IEEE Transactions on Cybernetics}
\bvolume{50}(\bissue{12}),
\bfpage{4876}--\blpage{4885}
(\byear{2019})
\end{barticle}
\endbibitem

\bibitem{Parise2020Distributed}
\begin{barticle}
\bauthor{\bsnm{Parise}, \binits{F.}},
\bauthor{\bsnm{Gentile}, \binits{B.}},
\bauthor{\bsnm{Lygeros}, \binits{J.}}:
\batitle{A distributed algorithm for almost-{N}ash equilibria of average
  aggregative games with coupling constraints}.
\bjtitle{IEEE Transactions on Control of Network Systems}
\bvolume{7}(\bissue{2}),
\bfpage{770}--\blpage{782}
(\byear{2020})
\end{barticle}
\endbibitem

\bibitem{Yi2016Initialization}
\begin{barticle}
\bauthor{\bsnm{Yi}, \binits{P.}},
\bauthor{\bsnm{Hong}, \binits{Y.}},
\bauthor{\bsnm{Liu}, \binits{F.}}:
\batitle{Initialization-free distributed algorithms for optimal resource
  allocation with feasibility constraints and its application to economic
  dispatch of power systems}.
\bjtitle{Automatica}
\bvolume{74}(\bissue{12}),
\bfpage{259}--\blpage{269}
(\byear{2016})
\end{barticle}
\endbibitem

\bibitem{Facchinei2003Finite}
\begin{bbook}
\bauthor{\bsnm{Facchinei}, \binits{F.}},
\bauthor{\bsnm{Pang}, \binits{J.}}:
\bbtitle{Finite-Dimensional Variational Inequalities and Complementarity
  Problems}.
\bsertitle{Operations Research}.
\bpublisher{Springer},
\blocation{New York}
(\byear{2003})
\end{bbook}
\endbibitem

\bibitem{Godsil01}
\begin{bbook}
\bauthor{\bsnm{Godsil}, \binits{C.}},
\bauthor{\bsnm{Royle}, \binits{G.F.}}:
\bbtitle{Algebraic Graph Theory}.
\bsertitle{Graduate Texts in Mathematics},
vol. \bseriesno{207}.
\bpublisher{Springer},
\blocation{New York}
(\byear{2001})
\end{bbook}
\endbibitem

\bibitem{Horn2013Matrix}
\begin{bbook}
\bauthor{\bsnm{Horn}, \binits{R.A.}},
\bauthor{\bsnm{Johnson}, \binits{C.R.}}:
\bbtitle{Matrix Analysis},
\bedition{2}nd edn.
\bpublisher{Cambridge University Press},
\blocation{Cambridge}
(\byear{2013})
\end{bbook}
\endbibitem

\bibitem{Nedich2017Achieving}
\begin{barticle}
\bauthor{\bsnm{Nedi{\'c}}, \binits{A.}},
\bauthor{\bsnm{Olshevsky}, \binits{A.}},
\bauthor{\bsnm{Shi}, \binits{W.}}:
\batitle{Achieving geometric convergence for distributed optimization over
  time-varying graphs}.
\bjtitle{SIAM Journal on Optimization}
\bvolume{27}(\bissue{4}),
\bfpage{2597}--\blpage{2633}
(\byear{2017})
\end{barticle}
\endbibitem

\bibitem{Liang2019Exponential}
\begin{barticle}
\bauthor{\bsnm{Liang}, \binits{S.}},
\bauthor{\bsnm{Wang}, \binits{L.}},
\bauthor{\bsnm{Yin}, \binits{G.}}:
\batitle{Exponential convergence of distributed primal-dual convex optimization
  algorithm without strong convexity}.
\bjtitle{Automatica}
\bvolume{105},
\bfpage{298}--\blpage{306}
(\byear{2019})
\end{barticle}
\endbibitem

\end{thebibliography}

\end{document}